\documentclass[a4paper,12pt]{amsart}
\usepackage{amssymb,mathrsfs,amssymb,xcolor}
\usepackage{mathtools}
\usepackage[top=3cm,bottom=3cm,outer=3cm,inner=3cm,marginpar=3cm]{geometry}
\usepackage[utf8]{inputenc}
\usepackage[bookmarks,bookmarksnumbered]{hyperref}
\hypersetup{colorlinks = true, linkcolor = blue, anchorcolor = red, citecolor = blue, filecolor = red, urlcolor = red,
            pdfauthor=author}
\newtheorem{theorem}{Theorem}[section]
\newtheorem{lemma}[theorem]{Lemma}
\newtheorem{corollary}[theorem]{Corollary}
\newtheorem{proposition}[theorem]{Proposition}

\theoremstyle{definition}
\newtheorem{definition}[theorem]{Definition}

\newtheorem{remark}[theorem]{Remark}

\numberwithin{equation}{section}

\def\CC{\mathbb{C}} % Complex C
\def\varep{\varepsilon}
\def\psh{\mathrm{PSH}}
\def\capa{\mathrm{Cap}}
\def\reg{\mathrm{reg}}
\def\sing{\mathrm{sing}}

\newcommand{\paren}[1]{\left(#1\right)}
\newcommand{\bparen}[1]{\left[ #1\right]}
\newcommand{\set}[1]{\left\{#1\right\}}

\newcommand{\ol}[1]{\overline{#1}}

\begin{document}
%------------------------------------------------------------------
\title[Continuity of solutions to complex Monge-Amp\`{e}re equations]{Continuity of solutions to complex Monge-Amp\`{e}re equations on compact K\"{a}hler spaces}

\author{Ye-Won Luke Cho}
\address{Research Institute of Molecular Alchemy, Gyeongsang National University, Jinju, 52828, Republic of Korea}
\email{ywlcho@gnu.ac.kr}

\author{Young-Jun Choi}
\address{Department of Mathematics and Institute of Mathematical Science, Pusan National University, 2, Busandaehak-ro 63beon-gil, Geumjeong-gu, Busan, 46241, Republic of Korea}
\email{youngjun.choi@pusan.ac.kr}


\begin{abstract}
We prove the continuity of bounded solutions to complex Monge-Amp\`{e}re equations on reduced, locally irreducible compact K\"{a}hler spaces. This in particular implies that any singular K\"{a}hler-Einstein potentials constructed in \cite{EGZ09} and \cite{Tsuji88, TianZhang06, ST17} are continuous. We also provide an affirmative answer to a conjecture in \cite{EGZ09} by showing that a resolution of any compact normal K\"{a}hler space satisfies the continuous approximation property. Finally, we settle the continuity of the potentials of the weak K\"{a}hler-Ricci flows \cite{ST17, GLZ20} on compact K\"{a}hler varieties with log terminal singularities.
\end{abstract}

\maketitle

\section{Introduction}
Any complex space in this article will be assumed to be reduced, locally irreducible, and of pure dimension $n\geq1$. We denote by $X^\reg$ the regular locus of a complex space $X$ and by $X^\sing:=X\setminus X^{\textup{reg}}$ the singular locus of $X$.

Recall that each point $p\in X$ admits a local embedding $j_p:V\rightarrow\CC^N$, where $V$ is an open neighborhood of $p$ in $X$ and $N=N(p)\geq 1$ is an integer.

%Let $X$ be a complex analytic space.  We denote by $X^{\textup{reg}}$ the regular locus of $X$ and by $X^{\textup{sing}}:=X\setminus X^{\textup{reg}}$ the singular locus of $X$. In this article, we only consider reduced, locally irreducible complex analytic spaces of pure dimension $n\geq 1$.
\begin{definition}
\normalfont
Let $U$ be an open subset of $X$. A function $u:U\to \mathbb{R}\cup \{-\infty\}$  satisfying $u\not\equiv -\infty$ on $U$ is $(\textit{strictly})$ $\textit{plurisubharmonic}$ if, for each $p\in U$, there exist a local embedding $j_p:V\to \mathbb{C}^{N}$ and a (strictly) plurisubharmonic function $u'$ on an open neighborhood of $j_p(V)$ in $\mathbb{C}^{N}$ such that $u|_{V}=u'|_{j_p(V)}\circ j_p$. 
\end{definition}

Similarly, the notion of smooth functions and pluriharmonic functions on open sets in $X$ can  be well-defined. The set of (strictly) plurisubharmonic functions on $U$ will be denoted as $(\textup{S})\psh(U)$. 

 If $u\in \psh\cap L_{\textup{loc}}^{\infty}(U)$, then there exists a well-defined Radon measure  $(dd^cu)^n$ on $U^{\textup{reg}}$ associated with $u$ \cite{BedfordandTaylor82}. Furthermore, the integral
\begin{equation*}
 \int_{K}(dd^cu)^n:=\int_{K\cap U^{\textup{reg}}}(dd^cu)^n
\end{equation*}
 turns out to be finite for any compact Borel set $K\subset U$ \cite{Bedford82}. So the measure on $U^{\textup{reg}}$ extends to a Radon measure  on $U$ that vanishes on $U^{\textup{sing}}$. 

\begin{definition}\label{def Kahler form}
Let $\{U_{\alpha}\}$ be an open cover of $X$. A $\textit{K\"{a}hler form}$ on $X$ is a family $\omega=\{\rho_{\alpha}\in \textup{SPSH}\cap C^{\infty}(U_{\alpha})\}$ such that each $\rho_{\alpha}-\rho_{\beta}$ is pluriharmonic on $U_{\alpha}\cap U_{\beta}$. The pair $(X,\omega)$ is called a $\textit{K\"{a}hler space}$.  A function $u:X\to \mathbb{R}\cup \{-\infty\}$  is  $\omega$-$\textit{plurisubharmonic}$ if $\rho_{\alpha}+u$ is plurisubharmonic on $U_{\alpha}$ for each $\alpha$.
\end{definition}

 Note that the set of local forms $\{dd^c(\rho_{\alpha}|_{U^{\textup{reg}}_{\alpha}})\}$ defines a smooth K\"{a}hler form on $X^{\textup{reg}}$. We denote by $\psh(\omega)$ the set of $\omega$-plurisubharmonic functions on $X$. 
 
If $u\in \psh(\omega)\cap L^{\infty}(X)$, then the set of local Radon measures $\{(dd^c(\rho_\alpha+u))^n|_{U_{\alpha}}\}$ defines a Radon measure $(\omega+dd^cu)^n$ on $X$. Then consider the following complex Monge-Amp\`{e}re equation for $u$ on a compact K\"{a}hler space $(X,\omega)$:
\begin{equation}\label{MA equation}
(\omega+dd^cu)^n=f\omega^n,~u\in \psh(\omega)\cap L^{\infty}(X),
\end{equation}
where $f$ is a nonnegative function in $L^p(X,\omega^n)$ for some $p>1$. The terms on both sides of (\ref{MA equation}) are to be interpreted as finite Radon measures on $X$. A singular K\"{a}hler-Einstein potential generating a canonical metric on a compact klt pair, if any exists on the pair, satisfies the equation; see \cite{Yau78, EGZ09, GGZ23}.

\begin{theorem}\label{main theorem}
Let $(X,\omega)$ be a reduced, locally irreducible compact K\"{a}hler space. If $u\in \psh(\omega)\cap L^{\infty}(X)$ is a solution to $\textup{(}\ref{MA equation}\textup{)}$, then $u$ is continuous on $X$.
\end{theorem}

The theorem in the case where $(X,\omega)$ is a K\"{a}hler manifold was settled in \cite{Kolo98, EGZ09}. Then the improvements of the methods of \cite{Kolo98} led to the proof of the H\"{o}lder continuity of the solutions on K\"{a}hler manifolds in \cite{Kolo08, DDGHKZ14}. For the case of general compact K\"{a}hler spaces, the optimal global regularity of the solutions is expected to be the H\"{o}lder continuity. The solutions are indeed locally H\"{o}lder continuous on the regular locus of the given space  \cite{DDGHKZ14} and they are further smooth on the locus when $f$ is so \cite{Paun08, EGZ09, BBEGZ19}. But their regularity near the singular locus remains to be studied further. 

 If $X$ is a compact singular projective variety and $\omega$ is a Hodge form on $X$, then Theorem \ref{main theorem} also follows from the global smooth approximation of $u$ \cite{CGZ13} which can be shown to be uniform by the $L^{\infty}$ estimate for (\ref{MA equation}) \cite{EGZ09}. This in particular implies that a singular K\"{a}hler-Einstein potential for a klt pair $(X,D)$ is continuous on $X$ if the $\mathbb{Q}$-Cartier divisor $K_X+D$ is ample or anti-ample, as a version of Kodaira's embedding theorem for compact normal K\"{a}hler spaces is available \cite{EGZ09}. It is also known that solutions to \eqref{MA equation} are continuous near isolated singularities of an arbitrary compact complex space $X$ \cite{GGZ23}.

 In this article, we generalize the methods in \cite{Kolo98} to compact K\"{a}hler spaces. Our proof is of local nature and it is independent of any cohomological property of $\omega$, projectivity of the given space, and type of the singularities of the space. Hence we obtain the following
 \begin{corollary}
 If  a compact klt pair $(X,D)$ admits a singular K\"{a}hler-Einstein metric, then any singular K\"{a}hler-Einstein potential generating the metric is continuous on $X$. 
 \end{corollary}

For the proof of Theorem \ref{main theorem}, the strict positivity of the base form $\omega$, local $L^{\infty}$ estimate for solutions to (\ref{MA equation}) (Theorem \ref{C^0 estimate}), and Hartogs' lemma on complex spaces (Lemma \ref{Hartogs lemma}) play crucial roles. Then a mild adjustment of the proof also yields the continuity of the singular K\"{a}hler-Einstein potentials on compact minimal projective varieties of general type with log terminal singularities \cite{Tsuji88, TianZhang06, ST17}; see Section \ref{sect. big and nef KX}. In Section \ref{section global approx.}, we show that Theorem \ref{main theorem} together with the method of plurisubharmonic envelopes developed in \cite{GLZ19} implies that a resolution of any compact normal K\"{a}hler space enjoys the global continuous approximation property. This in particular provides an affirmative answer to a conjecture of Eyssidieux-Guedj-Zeriahi (p.614 of \cite{EGZ09}).  Finally, we prove the continuity of the potential of the weak K\"{a}hler-Ricci flow \cite{ST17, GLZ20} on a compact K\"{a}hler variety with log terminal singularities in Section \ref{sect. KRF}.

\subsection*{Acknowledgements}
We would like to thank Kang-Tae Kim for his suggestion that the authors should study the regularity of the singular K\"{a}hler-Einstein potentials. We also thank S\l{}awomir Ko\l{}odziej for his helpful comments on the article and  Tsz On Mario Chan, S\l{}awomir Dinew, Ngoc Cuong Nguyen for the helpful discussions. The first named author would like to thank Kang-Hyurk Lee, Kyeong-Dong Park, and Juncheol Pyo for their support and words of encouragement. The work was supported by  G-LAMP (RS-2023-00301974), the National Research Foundation of Korea (NRF-2021R1A4A1032418, RS-2023-00246259, NRF-2023R1A2C1007227), Samsung Science and Technology Foundation (SSTF-BA2201-01).

\section*{Statements and Declarations}
\subsection*{Conflict of interests} The authors state that there is no conflict of interests.
\subsection*{Data availability statement} This article has not used any associated data.

\section{Pluripotential theory on complex analytic spaces}
\subsection{Local $L^{\infty}$ estimate for complex Monge-Amp\`{e}re operator}
\begin{definition}[\cite{Bedford82}]\label{definition_capacity}
Let $X$ be a complex analytic space and $\Omega$ a nonempty open subset of $X$. 
The \emph{Monge-Amp\`{e}re capacity} of a Borel set $E\subset \Omega$ is 
\begin{equation*}
	\capa(E,\Omega):=\sup
	\set{
		\int_{E}(dd^cu)^n:u\in\psh(\Omega),~-1\leq u\leq 0~\text{on}~\Omega
	}.
\end{equation*}
\end{definition}
 Whenever there is no room for confusion, we shall use the notation  $\capa(E):=\capa(E,\Omega)$. It is straightforward to check that the capacity satisfies the following properties:
\medskip
\begin{enumerate}
	\setlength\itemsep{0.1em}
	\item If $E_1,E_2\subset \Omega$, then
	\[
	\capa(E_1\cup E_2)\leq \capa(E_1)+\capa(E_2).
	\]
	\item If $\{E_j\}$ is an increasing sequence of Borel sets in $\Omega$, then
	\begin{equation}\label{capacity union}
		\capa\paren{\bigcup_{j=1}^{\infty}E_j}=\lim\limits_{j\to \infty}\capa\paren{E_j}.
	\end{equation}
\end{enumerate}
We refer the reader to \cite{Bedford82} for other properties of the capacity.

\medskip
We say that a sequence $\{u_j\}\subset \psh(\Omega)$  \emph{converges to} $u\in \psh(\Omega)$ \emph{with respect to capacity} if
\[
\lim\limits_{j\to \infty}\capa\paren{\set{z\in K:|u_j(z)-u(z)|>\varep}}=0
\]
holds for any $\varep>0$ and a relatively compact subset $K$ of $\Omega$. The following will be important in the proof of Theorem \ref{main theorem}.
\begin{proposition}\label{decreasing sequence}
 If a decreasing sequence $\{u_j\}\subset \psh\cap L^{\infty}(\Omega)$ converges pointwise to $u\in  \psh\cap L^{\infty}(\Omega)$ on $\Omega$, then the sequence $\{u_j\}$ converges to $u$ with respect to capacity.
\end{proposition}
\begin{proof}
Let $K$ be a relatively compact subset in $\Omega$ and $\varep, \varep'$ positive numbers.
For each $j\geq 1$, define 
\[
K_{j}:=\{z\in K:|u_j(z)-u(z)|>\varep\}.
\] 
By Equation (5) in \cite{Bedford82}, there exists an open neighborhood $U\subset \Omega$ of the compact set $K\cap X^{\textup{sing}}$ such that $\capa(U)<\varep'/2$. Applying Proposition 1.12 in \cite{Kolo05} to finitely many relatively compact subsets of the compact set $K-U\subset X^{\textup{reg}}$, one can also find a positive integer $j_0$ such that
\[
\capa(K_j-U)\leq \frac{\varep'}{2}
\]
whenever $j\geq j_0$. This implies that
\[
\capa(K_j)\leq \capa(U)+\capa(K_j-U)<\varep'
\]
if $j\geq j_0$.
\end{proof}
A domain $\Omega$ in a complex space $X$ is $\textit{strictly pseudoconvex}$ if it admits a negative $C^{2}$ strictly plurisubharmonic exhaustion function.  In the rest of the subsection, we shall assume that $(X,\omega)$ is a K\"{a}hler space and $\Omega$ is a relatively compact strictly pseudoconvex domain in $X$.
\begin{theorem}\label{C^0 estimate}
   Suppose that $u\in  \psh(\Omega)\cap L^{\infty}(\overline{\Omega})$ and $v\in  \psh(\Omega)\cap C^{0}(\overline{\Omega})$ satisfy $\liminf_{\Omega \ni \zeta \to z}\,(u(\zeta)-v(\zeta))\geq 0~\text{for each}~z\in \partial \Omega$. If
\begin{equation}\label{density assumption}
(dd^cu)^n= f\omega^n~\text{for some}~ f\in L^p(\Omega,\omega^n),\,p>1
\end{equation}
 holds on $\Omega$, then there exists a uniform constant $A=A(\|f\|_{L^p})>0$ such that
 \begin{equation}\label{sup estimate}
  \sup_{\Omega}\paren{v-u}\leq \varep+A\cdot\bparen{\capa(\{v-u>\varep\})}^{\frac{1}{n}}
 \end{equation}
for any $\varep>0$.
\end{theorem}

The proof of the theorem follows the arguments in \cite{GKZ08}. First, recall the following comparison principle on complex analytic spaces.

\begin{theorem}[\cite{Bedford82}]\label{comparison principle th}
If $u,v\in \psh\cap L^{\infty}(\Omega)$ satisfy
\[
\liminf_{\Omega \ni \zeta \to z}\,(u(\zeta)-v(\zeta))\geq 0~\text{for each}~z\in \partial \Omega,
\]
 then
	\[
	\int_{\{u<v\}}(dd^cv)^n\leq \int_{\{u<v\}}(dd^cu)^n.
	\]
\end{theorem}

The following three lemmas are crucial in the proof of Theorem \ref{C^0 estimate}.
\begin{lemma}\label{capacity lemma}
If $u,v\in \psh\cap L^{\infty}(\Omega)$ satisfy
\[
\liminf_{\Omega \ni \zeta \to z}\,(u(\zeta)-v(\zeta))\geq 0~\text{for each}~z\in \partial \Omega,
\]
 then for any $s>0$ and $t\geq 0$, we have
\[
t^n\capa(\{u-v<-s-t\})\leq \int_{\{u-v<-s\}}(dd^cu)^n.
\]
\end{lemma}
\begin{proof}
Choose a function $w\in \psh(\Omega)$ satisfying $-1\leq w\leq 0$ on $\Omega$. Then 
\begin{equation}\label{setinclusions}
\{u-v<-s-t\}\subset \{u<v-s+tw\}\subset \{u<v-s\}\subset \Omega.
\end{equation}
Note that each set above is relatively compact in $\Omega$. Since $(dd^c(u+v))^n\geq (dd^cu)^n$ holds for any $u,v\in \psh\cap L^{\infty}(\Omega)$ (see \cite{Klimek91}), we obtain
\begin{align*}
I(w)&:=t^n\int_{\{u-v<-s-t\}}(dd^cw)^n=\int_{\{u-v<-s-t\}}(dd^c(tw))^n\\
&\leq\int_{\{u<v-s-t\}}(dd^c(v-s+tw))^n\leq\int_{\{u<v-s+tw\}}(dd^c(v-s+tw))^n.
\end{align*}
So Theorem \ref{comparison principle th} and  (\ref{setinclusions}) imply
\begin{equation*}\label{integralestimate}
I(w)\leq \int_{\{u<v-s+tw\}}(dd^cu)^n\leq \int_{\{u<v-s\}}(dd^cu)^n.
\end{equation*}
This completes the proof.
\end{proof}

\begin{lemma}\label{cap. dom. lemma}
Fix $p>1$ and a nonnegative function $f \in L^p(\Omega,\omega^n)$. Then there exists a constant $D=D(\|f\|_{L^p})>0$ such that
	\begin{equation}\label{cap. domination ineq.}
		\int_{K}f\omega^n\leq D\cdot [\capa(K)]^{2}
	\end{equation}
	for any relatively compact Borel set $K\subset \Omega$.
\end{lemma}
\begin{proof}
	By the H\"{o}lder inequality, we have
	\[
	\int_{K}f\omega^n\leq \|f\|_{L^p}\cdot [\textup{Vol}_{\omega}(K)]^{\frac{1}{q}},
	\]
	where
	\[
	\|f\|_{L^p}:=\bigg(\int_{\Omega}f^p\omega^n\bigg)^{\frac{1}{p}},~q:=\bigg(1-\frac{1}{p}\bigg)^{-1},~\textup{Vol}_{\omega}(K):=\int_{K}\omega^n.
	\]
	Since $X$ is assumed to be reduced and locally irreducible, there exists a constant $C=C(n,\Omega)>0$ such that
	\begin{equation}\label{capacity volume comparison}
		\textup{Vol}_{\omega}(K)\leq C\cdot \textup{exp}(-\capa(K)^{-\frac{1}{n}})
	\end{equation}
     by Lemma 1.9 in \cite{GGZ23}.
	Therefore, we have
	\[
	\int_{K}f\omega^n\leq C^{\frac{1}{q}}\|f\|_{L^p}\cdot \textup{exp}\bigg(-\frac{1}{q}\,\capa(K)^{-\frac{1}{n}}\bigg).
	\]
	Now (\ref{cap. domination ineq.}) follows as $\textup{exp}(-t^{-\frac{1}{n}})=O(t^2)$ if $t\to 0^{+}$.
\end{proof}
Denote by $\mathbb{R}_{\geq 0}$ the set of nonnegative real numbers. For the sake of completeness, we recapitulate the proof of the following
\begin{lemma}[\cite{EGZ09}]\label{incr ftn lemma}
Let $g:\mathbb{R}_{\geq 0}\to \mathbb{R}_{\geq 0}$ be a decreasing right-continuous function satisfying $\lim_{x\to +\infty}g(x)=0$. Suppose that there exists a constant $B>0$ such that
\begin{equation}\label{equation}
t\cdot g(s+t)\leq B\{g(s)\}^{2}
\end{equation}
for any $s,t>0$. Then  $g(s_{\infty})=0$ for some finite number $s_{\infty}=s_{\infty}(B)>0$.
\end{lemma}

\begin{proof}
	Choose a real number $s_0>0$ so that $g(s_0)<(2B)^{-1}$. To find $s_{\infty}$, define a sequence $\{s_{j}\}\subset \mathbb{R}_{\geq 0}$ as follows: if $g(s_0)=0$, then we set $s_{\infty}=s_0$ to finish the proof. Otherwise, define
	\[
	s_1:=\inf\set{s>s_0:g(s)<\frac{g(s_0)}{2}}.
	\]
	Since $g$ is right-continuous, we have $g(s_1)\leq g(s_0)/2$. If $g(s_1)=0$, then we set $s_{\infty}=s_1$. Otherwise, we proceed inductively and define
	\[
	s_{j+1}:=\inf\bigg\{s>s_0:g(s)<\frac{g(s_j)}{2}\bigg\}.
	\]
    Note that the sequence $\{s_j\}$ enjoys the following properties:
	\begin{enumerate}
		\setlength\itemsep{0.1em}
		\item $s_{j+1}> s_j$  $\forall j\geq 0$.
		\item $g(s_{j+1})\leq \frac{1}{2}g(s_j)$  $\forall j\geq 0$.
		\item $0<\frac{1}{2}g(s_j)\leq g(s)\leq g(s_j)$ for any $j\geq 0$ and $s\in (s_j,s_{j+1})$.
	\end{enumerate}
	Fix $j\geq 0$ and let $s\in (s_j,s_{j+1})$. Then it follows from (\ref{equation}) that
	\[
	(s-s_j)\cdot g(s)\leq B\cdot \{g(s_j)\}^2\leq 2B\cdot g(s)\cdot g(s_j).
	\]
	Dividing the both sides of the inequality above by $g(s)>0$ and letting $s\to s_{j+1}$, we obtain
	\[
	s_{j+1}-s_j\leq 2B\cdot g(s_j)\leq \frac{2B\cdot g(s_0)}{2^j}
	\]
	for each $j\geq 0$. Define
	\begin{equation}\label{S infty estimate}
		s_{\infty}:=s_0+\sum_{j=0}^{\infty}(s_{j+1}-s_j)\leq s_0+4B\cdot g(s_0)<+\infty.
	\end{equation}
	Then  	$g(s_{\infty})\leq g(s_j)\leq 2^{-j}g(s_0)$ for any $j\geq 0$ so that $g(s_{\infty})=0$. 
\end{proof}
\textit{Proof of Theorem \ref{C^0 estimate}}. Fix $\varep>0,\,s\geq 0$ and define
\[
g(s):=[\capa(\{v-u>s+\varep\},\Omega)]^{\frac{1}{n}}.
\]
Then $g:\mathbb{R}_{\geq 0}\to \mathbb{R}_{\geq 0}$ is right-continuous by (\ref{capacity union}) and $\lim_{s\to +\infty}g(s)=0$ as $u,v$ are bounded on $\Omega$. Since $\liminf_{\Omega \ni \zeta \to z}\,(u(\zeta)-v(\zeta))\geq 0$ for each $z\in \partial \Omega$, the Borel set
\[
\{z\in \Omega:v(z)-u(z)>s \}
\]
is relatively compact in $\Omega$. Hence by Lemma \ref{capacity lemma} and Lemma \ref{cap. dom. lemma}, the function $g$ satisfies the assumptions in Lemma \ref{incr ftn lemma} if we let $B:=D^{\frac{1}{n}}$. So there exists a number $s_{\infty}=s_{\infty}(\|f\|_{L^p})>0$ such that
$
\capa(\{v-u>s_{\infty}+\varep\})=0.$ 
Then (\ref{capacity volume comparison}) implies $v-u\leq s_{\infty}+\varep$ on $\Omega-B'$ for some Borel set $B'\subset \Omega$ satisfying $\textup{Vol}_{\omega}(B')=0$. By the upper semicontinuity of $u$ and the continuity of $v$, we obtain 
\begin{equation}\label{C^0 estimate 1}
\sup_{\Omega}\,(v-u)\leq s_{\infty}+\varep.
\end{equation}
This estimate together with  (\ref{S infty estimate}) yields
\begin{equation*}
\sup_{\Omega}\,(v-u)\leq 4B\cdot[\capa(\{v-u>s_0+\varep\})]^{\frac{1}{n}}+s_0+\varep.
\end{equation*}
 If $g(0)<(2B)^{-1}$, then it follows from the proof of Lemma \ref{incr ftn lemma} that one can set $s_0=0$ in the inequality above. So we obtain the desired estimate. If $g(0)\geq (2B)^{-1}$, then (\ref{C^0 estimate 1}) implies
\[
\sup_{\Omega}\,(v-u)\leq 2Bs_{\infty}\cdot \frac{1}{2B}+\varep\leq 2Bs_{\infty}\cdot g(0)+\varep=2Bs_{\infty} [\capa(\{v-u>\varep\})]^{\frac{1}{n}}+\varep.
\]
 \hfill $\Box$

\subsection{Hartogs' lemma on complex analytic spaces}
By a $\textit{resolution}$ of a complex space $X$, we mean a proper bimeromorphic morphism $\pi:\tilde{X}\to X$, $\tilde{X}$ being a complex manifold, such that the map $\pi: \pi^{-1}(X^{\textup{reg}})\to X^{\textup{reg}}$ is biholomorphic and $\pi^{-1}(X^{\textup{sing}})$ is a simple normal crossing divisor. A complex space always admits a resolution \cite{Hi64, Wlo05}.

Denote by $\Delta$ the open unit disc in $\mathbb{C}$ centered at the origin. An upper semicontinuous function $u:X\to \mathbb{R}\cup \{-\infty\}$ is said to be $\textit{weakly plurisubharmonic}$ on $X$ if $u\circ f$ is plurisubharmonic on $\Delta$ for any holomorphic function $f:\Delta \to X$. By \cite{FornaessNarasimhan80}, $u$ is plurisubharmonic on $X$ if, and only if, it is weakly plurisubharmonic on $X$. 
\begin{lemma}[cf. \cite{Hartogs1906}]\label{Hartogs lemma}
Let $\Omega$ be an open subset of a complex analytic space $X$. Suppose that a sequence $\{u_j\}\subset \psh(\Omega)$ is locally uniformly bounded from above and there exists a number $m\in \mathbb{R}$ such that 
\[
\limsup_{j\to \infty}u_j(z)\leq m
\]
for any $z\in \Omega$. Then for each $\varep>0$ and a relatively compact set $K\subset \Omega$, there exists a positive integer $j_0=j_0(\varep,K)$ such that $u_j(z)\leq m+\varep$ if $z\in K$ and $j\geq j_0$.

\end{lemma}

\begin{proof}
  Let $\pi:\tilde{\Omega}\to \Omega$ be a resolution of $\Omega$ and $\tilde{K}:=\pi^{-1}(K).$ Note that each function $\tilde{u}_j:=u_j\circ \pi$ is plurisubharmonic on $\tilde{\Omega}$ as it is weakly plurisubharmonic on $\tilde{\Omega}$. Then the sequence $\{\tilde{u}_j\}\subset \textup{PSH}(\tilde{\Omega})$ is locally uniformly bounded from above on $\tilde{\Omega}$ and it satisfies 
\[
\limsup_{j\to \infty}\tilde{u}_j(z)=\limsup_{j\to \infty}u_j(\pi(z))\leq m
\] 
for each $z\in \tilde{\Omega}$. Applying Hartogs' lemma \cite{Hartogs1906} to the sequence $\{\tilde{u}_j\}$ and finitely many relatively compact subsets of the compact set $\tilde{K}$, one can find a positive integer $j_0=j_0(\varep,K)$ such that $\tilde{u}_j(z)\leq m+\varep$ whenever $j\geq j_0$ and $z\in \tilde{K}$. Therefore, $u_j(z)\leq m+\varep$ if $z\in K$ and $j\geq j_0$ as desired.
\end{proof}

Given $z\in \Omega$, denote by $\mathcal{N}_z$ the set of open neighborhoods of $z$ in $\Omega$. The $\textit{lower semicontinuous}$ $\textit{regularization}$ of a function $u:\Omega\to (-\infty,+\infty]$ is a lower semicontinuous function $u_{\ast}:\Omega \to (-\infty,+\infty]$ defined as
 \begin{equation}\label{liminf}
u_{\ast}(z):=\sup_{V\in \mathcal{N}_{z}}\inf_{V}u
 \end{equation}
for each $z\in \Omega$. We shall need the following version of Lemma \ref{Hartogs lemma} in Section \ref{section for proof} (cf. Proposition 2.4.1 in \cite{Kolo98}).
\begin{corollary}\label{Hartog lemma 2}
	Let $\Omega$ and $X$ be as before and let $c>0,t>1$. Suppose that a sequence $\{u_j\}\subset \psh\cap C^0(\Omega)$ is uniformly bounded and it converges pointwise to a positive function $u\in \psh(\Omega)$. If $u-tu_{\ast}<c$ on a compact set $K\subset \Omega$, then there exist a positive integer $j_0$ and an open neighborhood $U\subset \Omega$ of $K$ such that
	\[
	u_j(z)<tu(z)+c~\text{for any}~z\in U~\text{and}~j\geq j_0
	\]
\end{corollary}
\begin{proof}
Fix $z_0\in K$. Then there exists a number $c_1\in (0,c)$ such that $u(z_0)-tu_{\ast}(z_0)<c_1$.
By the upper semicontinuity of $u$, one may choose an open  neighborhood $V$ of $z_0$ such that
$u(z)<tu_*(z_0)+c_1$ for each $z\in V$. It also follows from $(\ref{liminf})$ that 
\begin{equation*}
	u_*(z_0)-\frac{c_2-c_1}{t}<\inf_{V}u
\end{equation*}
for some $c_2\in (c_1,c)$, by shrinking $V$ if necessary. So we obtain
\begin{equation*}
u(z)<tu_*(z_0)+c_1<t\inf_Vu+c_2~\text{for any}~ z\in V.
\end{equation*}
This in turn implies
\begin{equation}\label{E:harnack_inequality}
\sup_V u < t\inf_V u+c_2.
\end{equation}
Let $W\subset V$ be a relatively compact open neighborhood of $z_0$ and recall that $\set{u_j}$ converges to $u$ pointwise. So by Lemma~\ref{Hartogs lemma} and \eqref{E:harnack_inequality},  there exists a positive integer $j_0$ depending on $\ol W$ and $c-c_2>0$ such that
\begin{equation*}
u_j(z)\leq \sup_Vu+(c-c_2)<t\inf_Vu+c\leq tu(z)+c
\end{equation*}
for any $z \in W$ and $j\geq j_0$.
Since $K$ is compact and $z_0\in K$ is arbitrary, one can find an open neighborhood $U$ of $K$ on which the inequality above holds after increasing $j_0$ if necessary. 
This completes the proof.
\end{proof}

\begin{remark}
Another proof of Lemma \ref{Hartogs lemma} in the case where $X$ is a Stein space can be found in \cite{AbduKamo23}. 
\end{remark}

\section{Proof of Theorem \ref{main theorem}}\label{section for proof}
Let $\varphi\in \psh(\omega)\cap L^{\infty}(X)$ be a solution to (\ref{MA equation}). We shall assume that $\varphi$ is discontinuous at some point on $X$ and derive a contradiction. Then we have
\begin{equation*}
d:=\textup{sup}_X\paren{\varphi-\varphi_*}>0
\end{equation*}
by the assumption. Since $\varphi-\varphi_*$ is a bounded upper semicontinuous function on the compact set $X$, the supremum is attained at some point $z_0\in X$. One may also assume further that
\begin{equation*}
	\varphi(z_0)
	=
	\min\set{\varphi(z):z\in X,\,\varphi(z)-\varphi_{\ast}(z)=d}
\end{equation*}
as the lower semicontinuous function $\varphi_{\ast}$ attains the minimum on the compact set $\{z\in X:\varphi(z)-\varphi_{\ast}(z)=d\}$ and $\varphi=\varphi_{\ast}+d$ holds on the set.

In this setting, we first settle the following

\begin{lemma}\label{L:minimum}
There exist open neighborhoods $B'\subset\subset B$ of $z_0$ in $X$ and $v\in C^\infty(B)$ such that $\omega=dd^cv$ on $B^\reg$ and 
\begin{equation*}
v(z_0)<v(z)~\text{for each}~ z\in B'\setminus\set{z_0}.
\end{equation*}
\end{lemma}
\begin{proof}
Let $j_{z_0}:B\to j_{z_0}(B)\subset \mathbb{C}^N$ be a local embedding of an open neighborhood $B$ of $z_0$ in $X$. By shrinking $B$ if necessary, one may choose $v\in \textup{SPSH}\cap  C^{\infty}(B)$  such that $\omega= dd^cv$ on $B^\reg$. Then by definition, there  exist an open neighborhood $U$ of $j_{z_0}(B)$ in $\mathbb{C}^N$ and $\tilde{v}\in  \textup{SPSH}\cap C^{\infty}(U)$ for which $\tilde{v}\circ j_{z_0} =v $ holds on $B$. Define a pluriharmonic function $h:\mathbb{C}^N\to \mathbb{R}$ as the real part of the function 
\[
f(z)= -2\sum_{k=1}^{N} z_k\frac{\partial \tilde{v}}{\partial z_k}(\tilde{z}_0)-\sum_{j,k=1}^{N}z_jz_k\frac{\partial^2 \tilde{v}}{\partial z_j\partial z_{k}}(\tilde{z}_0)
\]
holomorphic on $\mathbb{C}^N$ and let $w:=\tilde{v}+h\in \psh(U),$ $\tilde{z}_0:=j_{z_0}(z_0)$. 
Then for any $z=(z_1,\dots,z_N)\in \mathbb{C}^N$ satisfying $\tilde{z}_0+z\in U$, we have
\begin{align*}
        w(\tilde{z}_0+z)-w(\tilde{z}_0)
        =&
        2\cdot \textup{Re}
        \paren{
                    \sum_{k=1}^{N}z_k\frac{\partial w}{\partial z_k}(\tilde{z}_0)
                    +
                    \frac{1}{2}\sum_{j,k=1}^{N}z_jz_k\frac{\partial^2 w}{\partial z_j\partial z_k}(\tilde{z}_0)
        }
        \\
        &+
        \sum_{j,k=1}^{N}z_j\bar{z}_k\frac{\partial^2 w}{\partial z_j\partial \bar{z}_k}(\tilde{z}_0)+o(\|z\|^2)
        \\
        =&
        \sum_{j,k=1}^{N}z_j\bar{z}_k\frac{\partial^2 \tilde{v}}{\partial z_j\partial \bar{z}_k}(\tilde{z}_0)+o(\|z\|^2).
\end{align*}
This implies that $\tilde{z}_0$ has to be a strict local minimum point of $w$ since $\tilde{v}$ is strictly plurisubharmonic on $U$. So by replacing $v$ with $w\circ j_{z_0}$, one can find a relatively compact open neighborhood $B'\subset B$ of $z_0$ on which $z_0$ is a strict minimum point of $v$. Furthermore, we have 
\begin{equation*}
dd^c(w\circ j_{z_0})=dd^c(v+h\circ j_{z_0})=dd^cv=\omega~\text{on}~ B^\reg
\end{equation*}
as $h$ is pluriharmonic on $\mathbb{C}^N$. This completes the proof.
\end{proof}
Then it follows from (\ref{MA equation}) that the function $u:=v+\varphi\in\psh\cap L^{\infty}(B)$ satisfies (\ref{density assumption}) on $B$. 
We may assume that $u>0$ on $B$ and $u(z_0)>d$ by adding a constant to $u$ as $\varphi$ is bounded on $X$. Moreover,  $j_{z_0}(B)$ can be chosen to be an intersection of a complex analytic space and some open ball in $\mathbb{C}^N$ by shrinking $B$ if necessary.
Then $B$ is strictly pseudoconvex so there exists a sequence $\{u_j\}\subset  \psh \cap C^{\infty}(B)$ that decreases to $u$ pointwise on $B$ \cite{FornaessNarasimhan80}. By Lemma~\ref{L:minimum}, we also have 
\begin{equation*}
\inf_{S'}v-v(z_0)>0,~S':=\partial B'.
\end{equation*}

For each $a\in [0,d]$, define
\begin{equation*}
E(a):=\{z\in \overline{B'}: \varphi(z)-\varphi_{\ast}(z)\geq d-a\}\ni z_0,~E:=E(0),
\end{equation*}
and
\begin{equation*}
c(a):=\varphi(z_0)-\inf_{E(a)}\varphi\geq 0.
\end{equation*}
Then the set $E(a)$ is compact for any $a\in [0,d]$ and $c(a)$ is a nondecreasing function. We claim that $\lim_{j\to \infty}c(1/j)=0$. Since $\varphi_{\ast}$ is lower semicontinuous on each compact set $E(1/j)$, there exists a point $x_j\in E(1/j)$ such that $\inf_{E(1/j)}\varphi_{\ast}=\varphi_{\ast}(x_j)$ for each $j\geq 1$. Let $x\in X$ be the accumulation point of $\{x_j\}.$ Then
\begin{align*}
\varphi(x)&\geq \limsup_{j\to \infty}\varphi(x_j)\geq\liminf_{j\to \infty}\varphi(x_j)\geq \liminf_{j\to \infty}\{\varphi_{\ast}(x_j)+d-1/j\}\geq \varphi_{\ast}(x)+d
\end{align*}
so that $x\in E$. Note also that $\varphi_{\ast}(x)\geq \varphi_{\ast}(z_0)$ by the choice of $z_0$. Therefore,
\begin{align*}
	\limsup_{j\to \infty}c(1/j)&= \varphi(z_0)-\liminf_{j\to \infty}\inf_{E(\frac{1}{j})}\varphi\leq \varphi(z_0)-\liminf_{j\to \infty}\inf_{E(\frac{1}{j})}\varphi_{\ast}-d\\ 
	& =  \varphi_{\ast}(z_0)-\liminf_{j\to \infty}\varphi_{\ast}(x_j)\leq \varphi_{\ast}(z_0)-\varphi_{\ast}(x)\leq 0
\end{align*}
and this settles the claim. 
Define
\begin{equation*}
A:=u(z_0)>d\quad\text{and}\quad b:=\inf_{S'}v-v(z_0)>0.
\end{equation*}
By the previous arguments, one may choose a real number $a_0$ satisfying
\begin{equation}\label{inequality 1}
0<a_0<d,~\text{and}~c(a)<\frac{b}{3}~\text{for any}~a\leq a_0.
\end{equation}
 Also, choose $t>1$ so that
\begin{equation}\label{inequality 2}
(t-1)(A-d)<a_0<(t-1)(A-d+\frac{2}{3}b).
\end{equation}

\begin{lemma}\label{L:est}
There exist an open neighborhood $V$ of $S'$ and a positive integer $j_0$ such that
\begin{equation*}
u_j<tu+d-a_0\;\;\text{on}\;\;V
\end{equation*}
for each $j\geq j_0$.
\end{lemma}
\begin{proof}
Choose $y\in S'\cap E(a_0)$. Then it follows from \eqref{inequality 1} that
\begin{equation*}
u_{\ast}(y)\geq v(z_0)+b+\varphi(y)-d\geq v(z_0)+b+\varphi(z_0)-c(a_0)-d\geq A-d+\frac{2}{3}b.
\end{equation*}
Hence by (\ref{inequality 2}), $(t-1)u_{\ast}(y)>a_0$ and this in turn implies
\begin{equation*}
u(y)\leq u_{\ast}(y)+d<tu_{\ast}(y)+d-a_0.
\end{equation*}
By Corollary \ref{Hartog lemma 2}, there exists an open neighborhood $V_1$ of the compact set $S'\cap E(a_0)$ and a positive integer $j_1$ such that
\begin{equation}\label{est}
u_j<tu+d-a_0\
\end{equation}
holds on $V_1$ for each $j\geq j_1$.
On the other hand, since $E(a_0)\cap (S'\setminus V_1)=\emptyset,$ the inequality $u-tu_{\ast}\leq u-u_{\ast}<d-a_0$ holds on $S'\setminus V_1$. 
Applying Corollary \ref{Hartog lemma 2} again to the sequence $\{u_j\}$ and the compact set $S'\setminus V_1$, one can find an open neighborhood $V_2$ of $S'\setminus V_1$ and $j_2\geq 1$ such that \eqref{est} holds on $V_2$ for each $j\geq j_2$. Then we let $V:=V_1\cup V_2$ and $j_0:=\max\paren{j_1,j_2}$ to complete the proof.
\end{proof}

Now we finish the proof of Theorem \ref{main theorem}. The first inequality in (\ref{inequality 2}) reduces to
\begin{align*}
a_0>(t-1)(A-d)=(t-1)u_*(z_0)=tu_{\ast}(z_0)+d-u(z_0).
\end{align*}
So there exists a number $a_1>0$ such that
\begin{equation}\label{est3}
tu_{\ast}(z_0)+d-a_0<u(z_0)-a_1\leq u_j(z_0)-a_1
\end{equation}
for each $j\geq 1$.  By Lemma~\ref{L:est}, the set
\[
W(j,c):=\{z\in \overline{B'}: w(z)+c<u_j(z)\},\, w:=tu+d-a_0,
\]
 is a relatively compact open set in $B'$ for each $c>0,\,j\geq j_0$. Furthermore, (\ref{liminf}) and (\ref{est3}) imply that the set $W(j,c)$ is nonempty for any $c\in (0,a_1)$ and $j\geq 1$.
Hence $\sup_{B'}\,(u_j-w)>a_1/2$
holds for each $j\geq 1$. 
Then letting $\varep={a_1}/{4}$ in (\ref{sup estimate}) of Theorem \ref{C^0 estimate}, we obtain
\begin{equation}\label{capacity estimate}
A\cdot [\capa(W(j,\frac{a_1}{4}),B')]^{\frac{1}{n}}>\frac{a_1}{4}>0
\end{equation}
 for each $j\geq j_0$. Note also that 
 \[
 W(j,\frac{a_1}{4})\subset \{z\in \overline{B'}: u(z)+(d-a_0+\frac{a_1}{4})<u_j(z)\}
 \]
  for any $j\geq 1$. So by Proposition \ref{decreasing sequence}, we have $\lim_{j\to \infty}\capa(W(j,\frac{a_1}{4}),B')=0$ and this contradicts (\ref{capacity estimate}). Therefore, $\varphi$ is continuous on $X$ as desired.
\hfill $\Box$
\begin{remark}
The arguments above reduce to those in Section 2.4 of \cite{Kolo98} if the underlying set $X$ is a compact K\"{a}hler manifold.
\end{remark}

\section{Potentials on minimal projective varieties of general type}\label{sect. big and nef KX}
\begin{definition}
	Let $\{U_{\alpha}\}$ be an open cover of a complex space $X$. A $\textit{semi-K\"{a}hler}$ $\textit{form}$ on $X$ is a family $\omega=\{\rho_{\alpha}\in \textup{PSH}\cap C^{\infty}(U_{\alpha})\}$ such that each $\rho_{\alpha}-\rho_{\beta}$ is pluriharmonic on $U_{\alpha}\cap U_{\beta}$. A function $u:X\to \mathbb{R}\cup \{-\infty\}$  is  $\omega$-$\textit{plurisubharmonic}$ if $\rho_{\alpha}+u$ is plurisubharmonic on $U_{\alpha}$ for each $\alpha$. Given a locally bounded $\omega$-plurisubharmonic function $u$ on $X$, the $\textit{semi-K\"{a}hler current}$ on $X$ associated with $u$ is defined to be $\omega+dd^cu:=\{\rho_{\alpha}+u\in \textup{PSH}\cap L^{\infty}_{\textup{loc}}(U_{\alpha})\}$.
\end{definition}
We denote by $\psh(\omega)$ the set of $\omega$-plurisubharmonic functions on $X$. Note that a semi-K\"{a}hler current defines a positive current on $X^{\textup{reg}}$.  As before, a well-defined Radon measure $(\omega+dd^cu)^n$ on $X$ can be associated with each $u\in \textup{PSH}(\omega)\cap L^{\infty}_{\textup{loc}}(X)$.

Let $(X,\omega_X)$ and $(Y,\omega_Y)$ be compact normal K\"{a}hler spaces and $\{U_{\alpha}\}$ an open cover of $Y$ such that $\omega_Y=\{\rho'_{\alpha}|_{U_{\alpha}}\}$. If $F:X\to Y$ is a surjective holomorphic map, then the family
\[
F^{\ast}\omega_Y:=\{(\rho_\alpha'\circ F)|_{F^{-1}(U_\alpha)}\}
\]
becomes a semi-K\"{a}hler form $\omega$ on $X$. We also assume that the map
\[
F:X\setminus V\to Y\setminus F(V)
\]
is an isomorphism for some analytic subset $V$ of $X$. Then $n:=\textup{dim}_{\mathbb{C}}X=\textup{dim}_{\mathbb{C}}Y$ and it follows from Remmert's Proper Mapping Theorem that $F(V)$ is an analytic subset of $Y$. In this setting, we show that the assumption on the strict positivity of the base form $\omega$ in Theorem \ref{main theorem} can be weakened  as follows.
\begin{theorem}[cf. \cite{DinewZhang10}]\label{thm CMA semipositive}
	If $f$ is a nonnegative function on $X$ satisfying $f\in L^p(X,\omega^n)$ for some $p>1$, then any solution $u\in  \textup{PSH}(\omega)\cap L^{\infty}(X)$ to the equation 
	\begin{equation}\label{eq CMA eqn semiposi}
		(\omega+dd^cu)^n= f\omega^n~\text{on}~X
	\end{equation}
	is continuous on $X$.
\end{theorem}

\begin{proof}
	Note first that $F$ induces a biholomorphic mapping between $X^{\textup{reg}}\setminus V$ and $Y^{\textup{reg}}\setminus F(V)$. So one may define
	\[
	g:=f\circ F^{-1}: Y^{\textup{reg}}\setminus F(V)\to \mathbb{R}.
	\]
	Then we have
	\begin{align*}
		\int_{Y}g^p\omega_Y^n&=\int_{Y^{\textup{reg}}\setminus F(V)}g^p\omega_Y^n=\int_{Y^{\textup{reg}}\setminus F(V)}(f\circ F^{-1})^p\omega^n_{Y}\\
		&=\int_{X^{\textup{reg}}\setminus V}f^p\omega^n= \int_{X}f^p\omega^n<+\infty
	\end{align*}
	since $X^{\textup{sing}}, Y^{\textup{sing}}$, $V$, and $F(V)$ are locally pluripolar. Similar computations also yield
	\begin{align*}
		\int_{Y}g\omega_Y^n&=\int_{X}f\omega^n=\int_{X}\omega^n=\int_{Y}\omega_Y^n.
	\end{align*}
	So by Theorem 6.3 in \cite{EGZ09}, there exists a function $v\in \textup{PSH}(\omega_Y)\cap L^{\infty}(Y)$ satisfying
	\[
	(\omega_Y+dd^cv)^n=g\omega_Y^n
	\]
	on $Y$. Then $v\in C^{0}(Y)$ by Theorem \ref{main theorem} and the pullback of the equation by the biholomorphic map $F:X^{\textup{reg}}\setminus V\to Y^{\textup{reg}}\setminus F(V)$ becomes
	\begin{align*}
		(\omega+dd^c(v\circ F))^n=F^{\ast}(\omega_Y+dd^cv)^n=F^{\ast}(g\omega_Y^n)=f\omega^n~\text{on}~X^{\textup{reg}}\setminus V.
	\end{align*}
 Hence $v\circ F\in C^0(X)$ is a solution to (\ref{eq CMA eqn semiposi}).
	
	If $u$ is a solution to (\ref{eq CMA eqn semiposi}), then it follows from the uniqueness of the solution to (\ref{eq CMA eqn semiposi})  that $u-v\circ F=C$ for some constant $C\in \mathbb{R}$; see Proposition 1.4 in \cite{EGZ09}. Therefore, we conclude that $u$ is continuous on $X$ as desired.
\end{proof}

Let $X$ be a compact ($\mathbb{Q}$-Gorenstein) minimal projective variety of general type with log terminal singularities. Then by \cite{Ka84}, there exists an integer $m\geq 1$ such that $mK_X$ is a Cartier divisor and the linear system $|mK_X|$ induces a holomorphic map $\Phi_m:X\to \mathbb{CP}^{N_m}$. The isomorphism type of the compact normal projective variety $\Phi_m(X)$ is independent of the choice of $m$ and $\textup{dim}_{\mathbb{C}}\Phi_m(X)=\textup{dim}_{\mathbb{C}}X$ holds. The compact normal variety $X_{\textup{can}}:=\Phi_m(X)$ is called the \emph{canonical model} of $X$.

Denote by $\omega_{\textup{FS}}$ the Fubini-Study metric on the complex projective space $\mathbb{CP}^{N_m}$. Then define a semi-K\"{a}hler form on $X$ as
\begin{equation*}\label{eqn: pull back of FS metric}
\omega:=\frac{1}{m}\Phi_{m}^{\ast}\omega_{\textup{FS}}.
\end{equation*}
Note that $\Phi_m$ induces a biholomorphic mapping between $X^{\circ}:=X^{\textup{reg}}\setminus \textup{Exc}(\Phi_{m})$ and $X^{\textup{reg}}_{\textup{can}}\setminus \Phi_m(\textup{Exc}(\Phi_{m}))$, where $\textup{Exc}(\Phi_{m})$ denotes the exceptional locus of $\Phi_{m}$. Therefore, $\omega$ defines a smooth K\"{a}hler metric on the open set $X^{\circ}$. The singular K\"{a}hler-Einstein potential $u\in \textup{PSH}(\omega)\cap L^{\infty}(X)$ generating the smooth canonical metric of negative Ricci curvature on $X^{\circ}$ satisfies (\ref{eq CMA eqn semiposi}); see \cite{Tsuji88, TianZhang06, ST17}. So we obtain the following 

\begin{corollary}\label{cor potential of min. mod.}
	If $X$ is a compact minimal projective variety of general type with log terminal singularities, then the singular K\"{a}hler-Einstein potential generating the smooth canonical metric of negative Ricci curvature on $X^{\circ}$ is continuous on $X$. 
\end{corollary}

\begin{remark}
	It is known that $X_{\textup{can}}$ for $X$ in Corollary \ref{cor potential of min. mod.} is a compact normal projective variety with log terminal singularities and $mK_{X_\textup{can}}$ is an ample Cartier divisor for some integer $m\geq 1$; see Section 1.3 in \cite{Kollar13}. Then by \cite{EGZ09}, $X_{\textup{can}}$ admits the singular K\"{a}hler-Einstein potential $u$ on $X_{\textup{can}}$ that generates the smooth canonical metric of negative Ricci curvature on $X^{\textup{reg}}_{\textup{can}}$. Furthermore, $u\in C^0(X_{\textup{can}})$ by \cite{CGZ13} so one can conclude that the singular K\"{a}hler-Einstein potential $\Phi_m^{\ast}u$ on $X$ is continuous. Although the argument provides a proof of Corollary \ref{cor potential of min. mod.} that is independent of Theorem \ref{thm CMA semipositive}, our proof depends neither on the algebro-geometric properties of $X_{\textup{can}}$ nor on the global smooth approximation property of $X_{\textup{can}}$ \cite{CGZ13} that in turn relies on the homogeneity of the ambient projective space.
\end{remark}
\section{Global approximation of $\omega$-plurisubharmonic functions}\label{section global approx.}
\subsection{Viscosity supersolutions and pluripotential supersolutions}
Let $X$ be a compact K\"{a}hler manifold of complex dimension $n$, $f:X\to \mathbb{R}$ a nonnegative continuous function, and $dV$ a smooth volume form on $X$. We shall assume that $\omega$ is a smooth semipositive closed (1,1)-form on $X$ satisfying $\int_X\omega^n>0$, although the statements in this subsection were formulated more generally in \cite{GLZ19} for $\omega$ in a big cohomology class on $X$ .

\begin{definition}
Let $u:X\to \mathbb{R}$ be a lower semicontinuous function on $X$. A \emph{lower test function} for $u$ at $x_0\in X$ is a function $q:U\to \mathbb{R}$ defined on an open neighborhood $U$ of $x_0$ such that $u\geq q$ on $U$ and $u(x_0)=q(x_0)$. 
\end{definition}
For a real-valued $C^2$ function $q$ on an open subset $U$ of $X$, define
\begin{equation*}
	(\omega+dd^cq)_{+}:=
	\begin{cases}
		\omega+dd^cq~ &\text{if }~ \omega+dd^cq\geq 0~\text{on}~U,\\
		0~&\text{otherwise}.
	\end{cases}
\end{equation*}
\begin{definition}
 A lower semicontinuous function $u:X\to \mathbb{R}$ is a \emph{viscosity supersolution} to the equation
 \begin{equation}\label{eqn. cano. polar.}
(\omega+dd^c\varphi)^n=e^{\varphi}fdV
 \end{equation}
on $X$ if, for each $x_0\in X$ and a $C^2$ lower test function $q$ of $u$ at $x_0$, the inequality
\[
(\omega+dd^cq)_+^n\leq e^qfdV
\]
holds at $x_0$. A function $\varphi\in \textup{PSH}(\omega)\cap L^{\infty}(X)$ is a \emph{pluripotential supersolution} to (\ref{eqn. cano. polar.}) if the Radon measure $(\omega+dd^c\varphi)^n$ satisfies $(\omega+dd^c\varphi)^n\leq e^{\varphi}fdV$ on $X$. 
\end{definition}

If $h:X\to \mathbb{R}$ is a Lebesgue measurable function that is bounded from below on $X$, then the \emph{pluripotential envelope} of $h$ is defined to be
\[
P(h):=(\textup{sup}\{u\in \textup{PSH}(\omega):u\leq h~\text{on}~X\})^{\ast},
\]
where the asterisk denotes the upper semicontinuous regularization of the function. Then we have the following

\begin{theorem}[\cite{GLZ19}]\label{thm viscosity}
If  a bounded lower semicontinuous function $h:X\to \mathbb{R}$ is a viscosity supersolution to \emph{(\ref{eqn. cano. polar.})} on $X$, then $P(h)\in \textup{PSH}(\omega)\cap L^{\infty}(X)$ is a pluripotential supersolution to \emph{(\ref{eqn. cano. polar.})} on $X$.	
\end{theorem}
\subsection{Smooth approximation property}
\begin{definition}
Let $X$ be a compact K\"{a}hler manifold and $\omega$ a smooth semipositive closed (1,1)-form on $X$. We say that $(X,\omega)$ satisfies $\textit{\textup{(continuous}\textup{)} smooth }$ $\textit{approximation property}$ if, for any $\omega$-plurisubharmonic function $\varphi$ on $X$, there exists a sequence of (continuous) smooth $\omega$-plurisubharmonic functions on $X$ that decreases pointwise to $\varphi$ on $X$. 
\end{definition}

 Any solution to (\ref{MA equation}) on a compact K\"{a}hler manifold $X$ is continuous on $X$ whenever $(X,\omega)$ enjoys the continuous approximation property; see Theorem 2.1 in \cite{EGZ09}. A pair $(X,\omega)$ satisfies the smooth approximation property if $\omega$ is K\"{a}hler  \cite{Demailly 92, Blocki Kolo 07} or a proper holomorphic map $\pi:X\to V$ is a resolution of a compact normal projective variety $V$  such that $\omega=\pi^{\ast}\omega_V$ for some Hodge form $\omega_V$ on $V$ \cite{CGZ13}. But it is not known whether a general pair $(X,\omega)$ satisfies the smooth approximation property.
 
 \begin{theorem}[cf. \cite{EGZ15}]
 If $(V,\omega_V)$ is a compact normal K\"{a}hler space and $\pi:X\to V$ is a resolution of $V$, then $(X,\pi^{\ast}\omega_V)$ satisfies the smooth approximation property.
 \end{theorem} 

\begin{proof}
Let $\omega:=\pi^{\ast}\omega_V$ and choose a function $\varphi\in \textup{PSH}(\omega)$. By Lemma 2.15 in \cite{Demailly 92}, there exists a sequence $\{h_j:X\to \mathbb{R}\}$ of smooth functions decreasing pointwise to $\varphi$ on $X$. Since $\varphi\leq P(h_j)\leq h_j$ holds for every $j\geq 1$, the decreasing sequence $\{\varphi_j:=P(h_j)\}\in \textup{PSH}(\omega)\cap L^{\infty}(X)$ converges pointwise to $\varphi$ on $X$. Then it suffices to show that each $\varphi_j$ is continuous on $X$ to complete the proof.

Let $j\geq 1$ and note that $(\omega+dd^ch_j)_{+}^n=f_jdV$ for some $f_j\in C^0(X)$. If $q:U\to \mathbb{R}$ is a $C^2$ lower test function of $h_j$ at $x_0\in X$, then $h_j-q$ attains a minimum on $U$ at $x_0$. So $dd^c(h_j-q)\geq 0$ at $x_0$ and
\[
(\omega+dd^ch_j)_{+}^n\geq (\omega+dd^cq)_{+}^n
\]
at $x_0$. This implies that $h_j$ is a viscosity supersolution to the equation
\begin{equation}\label{eqn viscosity}
(\omega+dd^c\varphi)^n=e^{\varphi}e^{-h_j}f_jdV
\end{equation} 
on $X$. By Theorem \ref{thm viscosity}, $\varphi_j$ is a pluripotential supersolution to (\ref{eqn viscosity}) on $X$. Hence we obtain $(\omega+dd^c\varphi_j)^n=g'_jdV$ on $X$ for a bounded function $g'_j\in L^{\infty}(X)$. Then it follows from Lemma 3.2 in \cite{EGZ09} that
\[
(\omega+dd^c\varphi_j)^n= g_j\omega^n~\text{on}~X,
\]
where $g_j\in L^p(X,\omega^n)$ for some $p>1$. Therefore each $\varphi_j$ is continuous on $X$ by Theorem \ref{thm CMA semipositive}. Then by Richberg's regularization theorem on complex analytic spaces, $(X,\pi^{\ast}\omega_V)$ satisfies the smooth approximation property; see, for example, Corollary 2.5 in \cite{EGZ15}.
\end{proof}

\section{Potentials of weak K\"{a}hler-Ricci flows}\label{sect. KRF}
\subsection{Weak K\"{a}hler-Ricci flows on K\"{a}hler varieties}
Let $X$ be a compact $\mathbb{Q}$-Gorenstein K\"{a}hler variety with log terminal singularities. Denote by $\mathcal{C}^{\infty}_X$ and $\mathcal{PH}_X$ the sheaf of smooth functions on $X$ and the sheaf of pluriharmonic functions on $X$, respectively. Then the short exact sequence of sheaves
\[
0\longrightarrow \mathcal{PH}_X\longrightarrow \mathcal{C}^{\infty}_X\longrightarrow \mathcal{C}^{\infty}_X/\mathcal{PH}_X\longrightarrow 0
\]
induces the boundary map
\[
 H^0(X,\mathcal{C}^{\infty}_X/\mathcal{PH}_X) \stackrel{[\cdot]}{\longrightarrow}
 H^1(X,\mathcal{PH}_X).
\]
The K\"{a}hler cone $\mathcal{K}$ is defined as
\[
\mathcal{K}:=\{[\omega]\in H^1(X,\mathcal{PH}_X): \omega~\text{is a K\"{a}hler form on}~X\}.
\]
We refer the reader to \cite{EGZ09, GGZ23} for the definition of the first Chern class map $c_1:H^1(X,\mathcal{O}^{\ast}_X)\to H^1(X,\mathcal{PH}_X)$. For a $\mathbb{Q}$-line bundle $L$ on $X$, let $c_1(L):=m^{-1}c_1(mL)$ where $m\geq 1$ is the smallest integer such that $mL$ is a holomorphic line bundle on $X$. Define $c_1(X):=-c_1(K_X)$.
\begin{definition}[\cite{ST17}] \label{def WKRF}
Let $\hat{\omega}_0$ be a K\"{a}hler form on $X$ such that
\[
T_{\textup{max}}:=\sup\{t>0:[\hat{\omega}_0]-tc_1(X)\in \mathcal{K}\}>0
\]
and fix $\varphi_0\in \textup{PSH}(\hat{\omega}_0)\cap L^{\infty}(X)$. A family $\{\omega(t):t\in [0,T_{\textup{max}})\}$ of semi-K\"{a}hler currents on $X$ is said to be a solution to the $\textit{weak K\"{a}hler-Ricci flow}$ starting from $\omega_0:=\hat{\omega}_0+dd^c\varphi_0$ if $\{\omega(t):t\in [0,T_{\textup{max}})\}$ restricts to a smooth path of K\"{a}hler forms on $X^{\textup{reg}}$ satisfying 
	\begin{equation}\label{eqn KRF}
		\begin{dcases}
			\frac{\partial \omega}{\partial t}=-\textup{Ric}(\omega)~~\text{on}~~(0,T_{\textup{max}})\times X^{\textup{reg}},\\[5pt]
			\omega|_{t=0}=\omega_0 ~~\text{on}~~X.
		\end{dcases}
	\end{equation}
\end{definition}

Let $T\in (0,T_{\textup{max}})$ be a positive number and $\hat{\omega}_T$ a K\"{a}hler form on $X$ such that $[\hat{\omega}_T]=[\hat{\omega}_0]-Tc_1(X)$. Define $\chi:=T^{-1}(\hat{\omega}_T-\hat{\omega}_0)$ and choose a smooth hermitian metric $h$ on the $\mathbb{Q}$-line bundle $K_X$ with the curvature form $\chi$. Then consider the adapted measure $\mu$ associated with $h$, and semi-K\"{a}hler forms
\[
\hat{\omega}_t:=\hat{\omega}_0+t\chi=\frac{1}{T}((T-t)\hat{\omega}_0+t\hat{\omega}_T),~t\in [0,T]
\]
on $X$. It is known that there exists a unique function $\varphi:[0,T)\times X\to \mathbb{R}$ satisfying

\begin{enumerate}
	\setlength\itemsep{0.7em}
	\item  $\varphi\in L^{\infty}([0,T]\times X)\cap C^{\infty}((0,T)\times X^{\textup{reg}})$, $\varphi\in \textup{PSH}(\hat{\omega}_t)\cap L^{\infty}(X) $ for each $t\in [0,T)$, and
	\item 
\begin{equation}\label{eqn parabolic CMA}
	\begin{dcases}
		\frac{\partial \varphi}{\partial t}=\textup{log}\,\frac{(\hat{\omega}_t+dd^c\varphi)^n}{ \mu}~~\text{on}~~(0,T)\times X^{\textup{reg}},\\[10pt]
		\varphi|_{t=0}=\varphi_0 ~~\text{on}~~X.
	\end{dcases}
\end{equation}
\end{enumerate}
Then $\omega(t):=\hat{\omega}_t+dd^c\varphi$ is a unique solution to the weak K\"{a}hler-Ricci flow starting from $\omega_0$ (see p.235 of \cite{Boucksom Guedj 13}). We shall recapitulate the construction of the solution \cite{Boucksom Guedj 13, ST17, GLZ20} for the sake of completeness.

Note first that it suffices to construct the solution on $[0,T)\times X$ for any given $T\in (0,T_{\textup{max}})$. Let $\pi:\tilde{X}\to X$ be a resolution of $X$, $dV$ a volume form on $\tilde{X}$, and $\tilde{\omega}$ a K\"{a}hler form on $\tilde{X}$. Define $\tilde{\mu}:=\pi^{\ast}\mu$, $\tilde{\omega}_t:=\pi^{\ast}\hat{\omega}_t$. Since $X$ has log terminal singularities, one may choose a number $p>1$ such that $\tilde{\mu}=\tilde{g}dV$ for some $\tilde{g}\in L^p(dV)$. Then we approximate
\begin{enumerate}
	\setlength\itemsep{0.1em}
	\item  the semipositive form $\tilde{\omega}_t$ by a sequence $\{\tilde{\omega}_j(t)=\tilde{\omega}_t+2^{-j}\tilde{\omega}:j\geq 1\}$ of K\"{a}hler forms on $\tilde{X}$ for each $t\in [0,T)$,
	\item $\tilde{g}\geq 0$ by a sequence $\{\tilde{g}_j:\tilde{X}\to \mathbb{R}\}$ of strictly positive smooth functions on $\tilde{X}$ that converges to $\tilde{g}$ in $L^p(dV)$, and
	\item $\tilde{\varphi}_0\in \textup{PSH}(\tilde{\omega}_0)$ by a decreasing sequence $\{\tilde{\varphi}_{0,j}\in \textup{SPSH}(\tilde{\omega}_j(0))\cap C^{\infty}(\tilde{X})\}$ (Lemma 2.15 in \cite{Demailly 92}).
\end{enumerate}

By Theorem 4.5.1 in \cite{Boucksom Guedj 13}, there exists a unique smooth function $\tilde{\varphi}_j:[0,T)\times \tilde{X}\to \mathbb{R}$ for each $j\geq 1$ such that $\tilde{\varphi}_{j}(t,\cdot)\in \textup{SPSH}(\tilde{\omega}_j(t))$ for any $t\in [0,T)$ and each $\tilde{\varphi}_j$ satisfies
\begin{equation}\label{eqn WKRF smooth approximants}
	\begin{dcases}
		\frac{\partial \tilde{\varphi}_{j}}{\partial t}=\textup{log}\,\frac{(\tilde{\omega}_{j}+dd^c\tilde{\varphi}_{j})^n}{ \tilde{g}_jdV}~\text{on}~[0,T)\times \tilde{X},\\[10pt]
		\tilde{\varphi}_j|_{t=0}=\tilde{\varphi}_{0,j}~\text{on}~\tilde{X}.
	\end{dcases}
\end{equation}
Since $\hat{\omega}_T$ is a K\"{a}hler form on $X$, $\tilde{\omega}_T\geq c \tilde{\omega}_0$ holds on $\tilde{X}$ for some constant $c\in (0,1)$. So we have
\begin{equation}\label{ineq. big form}
\tilde{\omega}_t=\tilde{\omega}_0+\frac{t}{T}(\tilde{\omega}_T-\tilde{\omega}_0)\geq \tilde{\omega}_0-(1-c)\frac{t}{T}\tilde{\omega}_0\geq c\tilde{\omega}_0
\end{equation}
for each $t\in [0,T)$. Then the a priori estimates in \cite{GLZ20} together with (\ref{ineq. big form}) imply that the sequence $\{\tilde{\varphi}_j\}$ converges locally uniformly to a function $\tilde{\varphi}:[0,T)\times \tilde{X} \to \mathbb{R}$ which is a unique weak solution to (\ref{eqn parabolic CMA}); see Theorem A and Theorem B in \cite{GLZ20}. Furthermore,  $\tilde{\varphi}$ is smooth on $(0,T)\times \tilde{X}^{\textup{reg}}$ by Theorem 4.1.4 and Lemma 4.4.4 in \cite{Boucksom Guedj 13}. Since $X$ is normal and $\tilde{\varphi}(t,\cdot)\in \textup{PSH}(\tilde{\omega}_t)$ for each $t\in [0,T)$, there exists a function $\varphi:[0,T)\times X\to \mathbb{R}$ such that $\varphi(t,\cdot)\in \textup{PSH}(\hat{\omega}_t)$ for any $t\in [0,T)$ and $\varphi\circ \pi = \tilde{\varphi}$ on $[0,T)\times \tilde{X}$; see the proof of Theorem 6.3 in \cite{EGZ09}. Then $\varphi$ is the desired solution.

Let $[\omega]\in H^{1,1}(X,\mathbb{R})$ be a big cohomology class on a compact K\"{a}hler manifold $X$. Then by definition, there exists a strictly positive current $T\in [\omega]$ with analytic singularities on $X$ that is smooth on a Zariski open set $\Omega$ in $X$. The $\textit{ample locus}$ $\textup{Amp}(\omega)$ of $[\omega]$ is defined to be the largest such Zariski open subset of $X$.
\begin{remark}\label{remark smoothness of WKRF}
	If $\hat{\omega}_0\in c_1([H])$ for a big, semiample $\mathbb{Q}$-Cartier divisor $H$ on a compact $\mathbb{Q}$-Gorenstein K\"{a}hler variety $X$ with log terminal singularities, then the existence of a unique weak solution to (\ref{eqn parabolic CMA}) can also be guaranteed as above. Furthermore, $\tilde{\varphi}$ is smooth on $(0,T)\times \textup{Amp}(\tilde{\omega}_0)$ by (\ref{ineq. big form}) and Lemma 4.4.4 in \cite{Boucksom Guedj 13}. If one assumes further that $X$ is a $\mathbb{Q}$-factorial compact projective variety, then by Kodaira's lemma (p.29 of \cite{Deb01}), there exists an effective $\mathbb{Q}$-Cartier divisor $\tilde{E}\subset \tilde{X}$ whose support is $\pi^{-1}(X^{\textup{sing}})$ such that the $\mathbb{Q}$-Cartier divisor $m\pi^{\ast}H-\tilde{E}$ is ample for any sufficiently large integer $m\geq 1$. So given a hermitian metric $h$ of the $\mathbb{Q}$-line bundle $[\tilde{E}]$, there exist positive numbers $c,\epsilon >0$ such that
	\[
	\tilde{\omega}_0+\epsilon dd^c\textup{log}\,h\geq c\tilde{\omega}~\text{on}~\pi^{-1}(X^{\textup{reg}}).
	\]
    Hence $\textup{Amp}(\tilde{\omega}_0)\supset \pi^{-1}(X^{\textup{reg}})$ holds and the weak solution to (\ref{eqn parabolic CMA}) is smooth on $(0,T_{\textup{max}})\times X^{\textup{reg}}$. The solution coincides with the one constructed in \cite{ST17} by Theorem 3.3 in \cite{ST17}.
\end{remark}
 \subsection{Lipschitz estimate for parabolic complex Monge-Amp\`{e}re equations.}
Let $(\tilde{X},\tilde{\omega})$ be a compact K\"{a}hler manifold and $T_{\textup{max}}>0$ a positive number. Choose a smooth family $\{\tilde{\omega}_t:t\in [0,T_{\textup{max}}]\}$ of K\"{a}hler forms on $\tilde{X}$ satisfying
\begin{equation*}\label{form comparison}
	\theta\leq \tilde{\omega}_t \leq \Theta ~\text{on}~[0,T_{\textup{max}}]\times \tilde{X},
\end{equation*}
where $\theta$ is a smooth closed semipositive (1,1)-form on $\tilde{X}$ whose cohomology class is big and $\Theta$ is a K\"{a}hler form on $\tilde{X}$. We also assume that $\tilde{\omega}_t=\tilde{\omega}_0+t\cdot \tilde{\chi}$ on $[0,T_{\textup{max}}]\times \tilde{X}$ for a smooth closed (1,1)-form $\tilde{\chi}$ on $\tilde{X}$.

\begin{proposition}[\cite{GLZ20}]\label{thm Lipschitz est.}
	Let $T\in (0,T_{\textup{max}}),~p>1$ be positive numbers and $\mu$ a finite measure on $\tilde{X}$ having smooth density $\tilde{g}>0$ with respect to a volume form $dV$ on $\tilde{X}$. If a smooth function $\tilde{\varphi}:[0,T]\times \tilde{X} \to \mathbb{R}$ satisfies $\tilde{\varphi}(t,\cdot)\in \textup{SPSH}(\tilde{\omega}_t)$ for each $t\in [0,T]$ and
	\begin{equation*}\label{eqn parabolic CMA 2}
		\frac{\partial \tilde{\varphi}}{\partial t}=\textup{log}\,\frac{(\tilde{\omega}_t+dd^c\tilde{\varphi})^n}{ \tilde{g}dV}~~\text{on}~~(0,T)\times \tilde{X},
	\end{equation*}
	then there exists a uniform constant $C$ depending only on $T, \,\theta,\, \Theta,\, \|\tilde{\varphi}(0,\cdot)\|_{L^{\infty}(\tilde{X})},$ and $\|\tilde{g}\|_{L^p(dV)}$ such that
	\begin{equation}\label{Lipschitz est}
		n\,\textup{log}\,t-C\leq \frac{\partial \tilde{\varphi}}{\partial t}(t,z)\leq \frac{C}{t}~\text{on}~(0,T)\times \tilde{X}.
	\end{equation}
\end{proposition}

 Then it follows from the proposition that each $\tilde{\varphi}_{j}$ in (\ref{eqn WKRF smooth approximants}) satisfies (\ref{Lipschitz est})  for some uniform constant $C>0$ independent of $j$. By the aforementioned construction of the potential $\varphi$ of the solution to the weak K\"{a}hler-Ricci flow, $\varphi$ is locally uniformly Lipschitz on $(0,T)\times X$ and there exists a constant $D>0$ such that
 \begin{equation}\label{ineq. supersol.}
(\hat{\omega}_t+dd^c\varphi)^n\leq e^{D}\mu
 \end{equation}
on $X$ for each $t\in (\frac{T}{2},T)$. Note also that the function
\begin{equation}\label{aux. function}
	t\in (0,T) \to \Phi(t,z):=\varphi(t,z)-n(t\,\textup{log}\,t-t)+Ct
\end{equation}
is nondecreasing for each $z\in X$.

\subsection{Continuity of potentials of weak K\"{a}hler-Ricci flows.}
By Proposition 3.12 in \cite{GLZ20}, the potential of the unique solution to the weak K\"{a}hler-Ricci flow on $X$ is continuous on  $(0,T_{\textup{max}})\times X^{\textup{reg}}$. We shall improve the regularity of $\varphi$ in the following form.

\begin{theorem}\label{thm KRF potential conti}
	The potential $\varphi:[0,T_{\textup{max}})\times X\to \mathbb{R}$ of the unique solution to the weak K\"{a}hler-Ricci flow on $X$ starting from $\omega_0=\hat{\omega}_0+dd^c\varphi_0$ is continuous on $(0,T_{\textup{max}})\times X$. If $\varphi_0$ is continuous on $X$, then $\varphi$ is continuous on $[0,T_{\textup{max}})\times X$.
\end{theorem}

\begin{proof}
Let $\varphi$ be the solution to (\ref{eqn parabolic CMA}) and choose a positive number $T\in (0,T_{\textup{max}})$. For each $t\in (0,T_{\textup{max}})$, there exists a number $q=q(t)>1$ such that 
\[
\frac{\mu}{\hat{\omega}_t^n}\in L^q(X,\hat{\omega}_t^n)
\]
as $X$ has log terminal singularities. Then by Theorem \ref{main theorem} and (\ref{ineq. supersol.}), the function $\varphi(t,\cdot)\in \textup{PSH}(\hat{\omega}_t)\cap L^{\infty}(X)$ is continuous on $X$ for any $t\in (0,T_{\textup{max}})$. Fix $(t_0,z_0)\in (0,T)\times X$ and let $(t,z)\in (0,T)\times X$. Since $\varphi(\cdot,z)$ is locally uniformly Lipschitz on $(0,T)$, we have
\begin{equation}\label{triangle inequaility}
	|\varphi(t,z)-\varphi(t_0,z_0)|\leq |\varphi(t,z)-\varphi(t_0,z)|+|\varphi(t_0,z)-\varphi(t_0,z_0)|\to 0
\end{equation}
if $(t,z)\to (t_0,z_0)$. Hence $\varphi$ is continuous on $(0,T)\times X$. This implies that $\varphi$ is continuous on $(0,T_{\textup{max}})\times X$ as $T\in (0,T_{\textup{max}})$ is arbitrary.

Suppose that $\varphi_0$ is continuous on $X$ and let $\Phi$ be the function in (\ref{aux. function}). Then by Theorem A in \cite{GLZ20}, we have
\[
\lim_{t\to 0^{+}}\Phi(t,z)=\lim_{t\to 0^{+}}\varphi(t,z)=\varphi_0(z)
\]
for each $z\in X$ and the convergence of $\Phi(\cdot,z)$ is uniform by Dini's theorem. So one may apply (\ref{triangle inequaility}) to $\Phi$ and conclude that $\Phi$ is continuous on $[0,T_{\textup{max}})\times X$. Therefore $\varphi$ is continuous on $[0,T_{\textup{max}})\times X$.
\end{proof}

Now we consider the weak K\"{a}hler-Ricci flow introduced in \cite{ST17}. Let $X$ be a $\mathbb{Q}$-factorial compact projective variety with log terminal singularities
and $H$ a big, semiample $\mathbb{Q}$-Cartier divisor on $X$. Then there exists an integer $m\geq 1 $ such that $mH$ is a Cartier divisor on $X$ and the linear system $|mH|$ induces a holomorphic map $\Phi_{mH}:X\to \mathbb{CP}^{N_m}$. Define a semi-K\"{a}hler form on $X$ as 
\[
\omega_{H}:=\frac{1}{m}(\Phi_{mH})^{\ast}\omega_{\textup{FS}}
\] 
and let $\mu$ be an adapted measure on $X$. For each $p\in (1,\infty]$, define
\begin{equation*}
	\textup{PSH}_p(X,\omega_H,\mu):=\bigg\{\varphi\in \textup{PSH}(\omega_H)\cap L^{\infty}(X):\frac{(\omega_H+dd^c\varphi)^n}{\mu}\in L^p(X,\mu)\bigg\},
\end{equation*}
\begin{equation*}
	\mathcal{K}_{H,p}(X):=\{\omega_{H}+dd^c\varphi:\varphi\in \textup{PSH}_p(X,\omega_H,\mu)\}.
\end{equation*}
By the arguments in the previous subsection, there exists a unique solution to the weak K\"{a}hler-Ricci flow on $X$ starting from a given current in $\mathcal{K}_{H,p}(X)$ (cf. Theorem 1.1 in \cite{ST17}). We may proceed as in the proof of Theorem \ref{thm KRF potential conti} to conclude that the potential of the solution to the flow is continuous on $[0,T_{\textup{max}})\times X$ once we have the following
\begin{lemma}\label{Lemma Holder}
	$\textup{PSH}_p(X,\omega_H,\mu)\subset C^0(X)$ for each $p\in (1,\infty]$.
\end{lemma}
\begin{proof}
	Let $\omega$ be a K\"{a}hler form on $X$ and $\varphi\in \textup{PSH}_p(X,\omega_H,\mu)$. By the assumption, there exists a function $f\in L^p(X,\mu)$ such that $(\omega_H+dd^c\varphi)^n=f\mu$ on $X$. Since $X$ has log terminal singularities, there also exist a number $q>1$ and $g\in L^{q}(X,\omega^n)$ such that $\mu=g\omega^n$ on $X$. Assume without loss of generality that $p=q$. So we have
	\begin{equation}\label{integrals}
		\int_{X}f^pd\mu=\int_{X}f^pg\omega^n<+\infty,~\text{and} \int_{X}g^p\omega^n<+\infty.
	\end{equation}
	Note that, if
	\begin{equation*}
		\frac{(\omega_H+dd^c\varphi)^n}{\omega^n}=fg\in L^{r}(X,\omega^n)
	\end{equation*}
	holds for some $r>1$, then the desired conclusion follows from Theorem \ref{thm CMA semipositive}. 
	
	Let $r>1$, $p'\in (1,p)$, and denote by $q'$ the H\"{o}lder conjugate of $p'$. Then we use the H\"{o}lder inequality to obtain 
	\begin{align*}
		0\leq \int_{X}(fg)^{r}\omega^n&=\int_{X}f^{r}g^{\frac{r}{p}}\cdot g^{\big(1-\frac{1}{p}\big)r}\omega^n\\
		&\leq \big\|f^{r}g^{\frac{r}{p}}\big\|_{L^{p'}(\omega^n)}\cdot \big\|g^{\big(1-\frac{1}{p}\big)r}\big\|_{L^{q'}(\omega^n)}\\
		&\leq \bigg[\int_{X}(f^{p}g)^{\frac{p'r}{p}}\omega^n\bigg]^{\frac{1}{p'}}\cdot \bigg[\int_{X}(g)^{\frac{p-1}{p'-1}\frac{p'r}{p}}\omega^n\bigg]^{\frac{1}{q'}}.
	\end{align*}
	Since $L^{p_1}(X,\omega^n)\subset L^{p_2}(X,\omega^n)$ whenever $p_1>p_2$, the estimate above together with (\ref{integrals}) implies that $fg\in L^{r}(X,\omega^n)$ if
	\[
	1+\frac{p-1}{p}<p'<p, ~1<r<\frac{p}{p'}
	\]
	holds.
\end{proof}

Let $X$ be a $\mathbb{Q}$-factorial compact projective variety with log terminal singularities and $\hat{\omega}_0$ a K\"{a}hler form on $X$ such that
	\[
	T_{m}:=\sup\{t>0:e^{-t}[\hat{\omega}_0]-(1-e^{-t})c_1(X)\in \mathcal{K}\}>0.
	\]
	Fix $\varphi_0\in \textup{PSH}(\hat{\omega}_0)\cap L^{\infty}(X)$. A family $\{\omega(t):t\in [0,T_m)\}$ of semi-K\"{a}hler currents on $X$ is said to be a solution to the $\textit{normalized weak K\"{a}hler-Ricci flow}$ starting from $\omega_0:=\hat{\omega}_0+dd^c\varphi_0$ if the family restricts to a smooth path of K\"{a}hler forms on $X^{\textup{reg}}$ satisfying 
	\begin{equation}\label{eqn NKRF}
		\begin{dcases}
			\frac{\partial \omega}{\partial t}=-\textup{Ric}(\omega)-\omega~~\text{on}~~(0,T_m)\times X^{\textup{reg}},\\[5pt]
			\omega|_{t=0}=\omega_0 ~~\text{on}~~X.
		\end{dcases}
	\end{equation}

If $X$ is a minimal variety of general type, then $mK_X$ is semiample for some integer $m\geq 1$ by \cite{Ka84}. So the weak K\"{a}hler-Ricci flow (\ref{eqn KRF}) with $T_{\textup{max}}=+\infty$ on $X$ admits a unique solution $\tilde{\omega}$ and the potential of the solution is continuous on $(0,+\infty)\times X$ by Theorem \ref{thm KRF potential conti}. Then the family of semi-K\"{a}hler currents on $X$ defined by
\[
\omega(t):=e^{-t}{\tilde{\omega}(e^t-1)}
\]
becomes a unique solution to (\ref{eqn NKRF}) with $T_m=+\infty$. By Proposition 5.4 in \cite{ST17}, the potential $\varphi(t,\cdot)$ of the solution to the normalized flow converges to the singular K\"{a}hler-Einstein potential in $L^1(X)\cap C^{\infty}(X^{\circ})$ as $t\to +\infty$. Therefore, we obtain the following
\begin{theorem}
	If $X$ is a compact minimal projective variety of general type with log terminal singularities, then the potential $\varphi:[0,+\infty)\times X\to \mathbb{R}$ of the unique solution to the normalized weak K\"{a}hler-Ricci flow on $X$ converging to the singular K\"{a}hler-Einstein potential in $L^1(X)\cap C^{\infty}(X^{\circ})$ is continuous on $(0,+\infty)\times X$. If $\varphi_0$ is continuous on $X$, then $\varphi$ is continuous on $[0,+\infty)\times X$.
\end{theorem}

\end{document}